\documentclass[a4papersize]{article}
\usepackage{amsmath,amsthm, amsxtra,amssymb,latexsym, amscd}
\usepackage[mathscr]{eucal}
\newtheorem{theorem}{Theorem}[section]
\newtheorem{corollary}[theorem]{Corollary}
\newtheorem{lemma}[theorem]{Lemma}

\newtheorem{example}[theorem]{Example}

\DeclareMathOperator{\depth}{depth}

\DeclareMathOperator{\Ann}{Ann}
\DeclareMathOperator{\docao}{ht}
\DeclareMathOperator{\Ass}{Ass}

\DeclareMathOperator{\Supp}{Supp}

\DeclareMathOperator{\Var}{Var}

\DeclareMathOperator{\nCM}{nCM}
\DeclareMathOperator{\Psupp}{Psupp}

\DeclareMathOperator{\Rad}{Rad}
\DeclareMathOperator{\Att}{Att}

\DeclareMathOperator{\Spec}{Spec}
\DeclareMathOperator{\p}{\frak p}
\DeclareMathOperator{\q}{\frak q}
\DeclareMathOperator{\m}{\frak m}
\DeclareMathOperator{\R}{\widehat R}

\begin{document}
\large
\centerline{\Large {\bf ON DEPTH OF MODULES IN AN IDEAL}}

\medskip

\vskip 0.7cm
\centerline { TRAN NGUYEN AN}
\centerline {Thai Nguyen University of Education, Thai Nguyen, Vietnam}
\centerline {e-mail: antn@tnue.edu.vn }
\medskip

\vskip 0.7cm

\noindent{\bf Abstract} {\footnote{ {\it{Key words and phrases: }} Pseudo support; Depth; Height of an ideal; Cohen-Macaulay local ring \hfill\break
  {\it{2000 Subject  Classification: }} 13D45, 13E15, 13E05. \hfill\break {\, This research is funded by Thai Nguyen University and Thai Nguyen University of Education under grant number DH2023-TN04-07}}. {Let $R$ be a commutative Noetherian ring, $I$ an ideal of $R$ and  $M$  a finitely generated $R$-module with $\dim_R(M)=d$. Denote by $\depth_R(I,M)$ the depth of $M$ in $I$. In \cite{HT}, C. Huneke and V. Trivedi proved that if $R$ is a quotient of a regular ring then there exists a finite subset $\Lambda_M $ of $\Spec(R)$ such that
  $$\depth_R(I,M)=\underset{\p\in \Lambda_M}{\min} \big\{ \depth_{R_{\p}}(M_{\p})+ \docao\big((I+\p)/\p\big) \big\}.$$
   Denote by $\Psupp^i_R(M)=\{\frak p\in\Spec(R)\mid H^{i-\dim(R/\frak p)}_{\frak p R_{\frak p}}(M_{\frak p})\neq 0\}$ the $i$-th pseudo support of $M$ defined by M. Brodmann and R. Y. Sharp \cite{BS1}. In this paper, we prove that if $\Psupp^i_R(M)$ is closed for all $i\leq d$ then  the above formula of $\depth_R(I,M)$ holds true, where $\Lambda_M =\underset{0\leq i\leq d}{\bigcup} \min \Psupp^i_R(M)$.  In particular, if $R$ is a quotient of a Cohen-Macaulay local ring then $\Lambda_M =\underset{0\leq i\leq d}{\bigcup}\min\Var\big(\Ann_R(H_{\m}^i(M))\big)$. We also give some examples to clarify the results.

\section{Introduction}

Throughout this paper, let $(R, \mathfrak m)$ be a Noetherian local ring and $M$ a finitely generated $R$-module of dimension $d$. For a subset $T$ of $\Spec (R),$ we denote by $\min T$ the set of all minimal elements of $T$ under the inclusion. For each ideal $I$ of $R$, denote by $\Var (I)$ the set of all prime ideals containing $I$. Following M. Brodmann and R. Y. Sharp \cite{BS1}, for each integer  $i \geq 0$, the {\it $i$-th pseudo support} of $M$, denoted by $\Psupp^i_R(M)$, is defined as follows $$\Psupp^i_R(M)=\{\p\in\Spec (R)\mid  H^{i-\dim (R/\p)}_{\p R_{\p}}(M_{\p})\neq 0\}.$$ 
Note that $\Psupp^i_R(M) \subseteq \Var(\Ann_R (H_{\m}^i(M)))$, cf. \cite[Lemma 2.3]{CNN}. If $R$ is a quotient of a Cohen-Macaulay local ring then $\Psupp^i_R(M) = \Var(\Ann_R (H_{\m}^i(M)))$, see \cite[Proposition 2.5]{BS1}. In this case, $\Psupp^i_R(M)$ is closed (under Zariski topology) for all $i$. The notion of pseudo support plays an important role in the study of the structure of Noetherian rings and modules, the dimension and multiplicity of local cohomology modules, the shifted principles for primes ideals and the Cohen-Macaulay locus, see \cite{BS1}, \cite{NA}, \cite{NA1}, \cite{CNN}. 

Let $I$ be a proper ideal of $R$. Denote by $\depth_R(I,M)$ the depth of $M$ in $I$. By C. Huneke and V. Trivedi \cite{HT},  
$$\depth_R(I,M)\leq \underset{\p\in\Supp_R(M)}{\min} \big\{ \depth_{R_{\p}}(M_{\p})+ \docao\big((I+\p)/\p\big) \big\}$$
and if $R$ is a quotient of a regular ring then there exists a finite subset $\Lambda_M $ of $\Supp_R(M)$ such that
$$\depth_R(I,M)=\underset{\p\in \Lambda_M}{\min} \big\{ \depth_{R_{\p}}(M_{\p})+ \docao\big((I+\p)/\p\big) \big\}.$$

  In case where $I=\frak m$, we have $$\depth_R(M) \leq \depth_{R_{\p}}(M_{\p}) +\dim(R/\p),$$ see \cite[Exercise 17.5]{Mat}. Moreover, for $t=\depth_R(M),$ it follows by \cite[Theorem 3.1]{An} that $\depth_R(M)=\depth_{R_{\p}}(M_{\p}) +\dim (R/\p)$ if and only if $\p\in\Psupp_R^t(M).$ Note that $\Psupp_R^t(M)\neq \emptyset$. Therefore,
$$\depth_R(M)=\underset{\p\in\Gamma}{\min} \{\depth_{R_{\p}}(M_{\p}) +\dim (R/\p)\}$$ for any non-empty finite subset $\Gamma$ of $\Psupp_R^t(M).$ 

In this paper, we use pseudo supports $\Psupp^i_R(M)$ to study the formula of $\depth_R(I,M)$ for any proper ideal $I$ of $R$. In case where $I\neq \frak m$, the $\depth_R(I,M)$ may be related to pseudo supports $\Psupp^i_R(M)$ for $i\leq d.$   We show that the above formula of  $\depth_R(I,M)$ holds true whenever all pseudo supports $\Psupp^i_R(M)$ are closed and $\Lambda_M =\underset{0\leq i\leq d}{\bigcup} \min \Psupp^i_R(M)$. In particular, if $R$ is a quotient of a Cohen-Macaulay local ring then we can choose $$\Lambda_M =\bigcup\limits_{i=0}^{d}\min\Var\big(\Ann_R(H_{\m}^i(M))\big).$$ 
The following theorem is the main result of this paper.

\begin{theorem}\label{T:1} For any proper ideal $I$ of $R$, the formula
	$$\depth_R(I,M)=\underset{\p\in \Lambda_M}{\min}\left\{ \depth_{R_{\p}}(M_{\p})+ \docao\big((I+\p)/\p\big) \right\}$$ 
	holds true if one of the following conditions is satisfied:
	
	{\rm (a)} $\Psupp^i_R(M)$ is closed for all $i\leq d$ and $\Lambda_M=\bigcup_{i=0}^d \min\Psupp^i_R(M)$; 

 {\rm (b)} $R$ is a quotient of a Cohen-Macaulay local ring and $\Lambda_M=\bigcup_{i=0}^d \min \Var(\Ann_R(H^i_{\m}(M)))$. 
\end{theorem}

 The proof of Theorem \ref{T:1} will be presented in the next section. Some examples are given to clarify the results of the paper, see Examples \ref{E:1}, \ref{E:2}, \ref{E:3}. In these examples, we compute concretely pseudo supports and annihilators of local cohomology modules, then we determine the set $\Lambda_M$ defined in Theorem \ref{T:1}.

\section{Proof of Theorem \ref{T:1} }

 We first need some properties of pseudo supports, see \cite[2.2, 2.5]{BS1}, \cite[Theorem 3.1]{NA}. Recall that  a subset $T$ of $\Spec(R)$ is said to be {closed under specialization} if for any $\frak p\subseteq \frak q$ with $\frak p, \frak q\in\Spec(R)$, if $\frak p\in T$ then $\frak q\in T.$ 

\begin{lemma} \label{L:1a} Let $i\geq 0$ be an integer. Then

{\rm (a)} $\Psupp^i_R(M)\subseteq \Var\big(\Ann_R (H_{\m}^i(M))\big).$

{\rm (b)} $\Psupp^i_R(M)$ is \textit{closed under specialization} if $R$ is catenary. In this case, $\Psupp^i_R(M)$ is closed if and only if $\min \Psupp^i_R(M)$ is a finite set.

{\rm (c)} $\Psupp^i_R(M) = \Var(\Ann_R (H_{\m}^i(M)))$  if $R$ is a quotient of a Cohen-Macaulay local ring. In this case, $\Psupp^i_R(M)$ is closed.
\end{lemma} 

Note that if $R$ is a Noetherian local domain of dimension $2$ then $\Psupp_R^0(R)=\emptyset,$ $\Psupp_R^1(R)\subseteq \{\frak m\}$ and $\Psupp_R^2(R)=\Spec(R)$, therefore $\Psupp_R^i(R)$ is closed for all $i.$ However, if $\dim (R)\geq 3$ then $\Psupp_R^i(R)$ is not necessarily closed, see \cite[Examples 3.1, 3.2]{BS1}.

Next, we give some relations among psupports and the depth of finitely generated modules, see \cite[Theorem 3.1(iii)]{CNN}, \cite[Theorem 3.1]{An}.

\begin{lemma}\label{L:1} For any integer $i\geq 0$ we have 
		
{\rm (a)} $\bigcup_{j=0}^i \Psupp_R^j(M)=\{\p \in \Supp_R(M)\mid \depth_{R_{\frak p}} (M_{\p}) +\dim (R/\p)\leq i\}.$

{\rm (b)} $ \Psupp_R^i(M) \setminus \bigcup_{j=0}^{i-1} \Psupp_R^j(M) =\{\p \in \Supp_R(M) \mid \depth_{R_{\p}}(M_{\p}) +\dim (R/\p) =i\}.$ 

{\rm (c)} Set $t:=\depth_R(M).$ Then $t= \depth_{R_{\p}}(M_{\p}) +\dim (R/\p)$ if and only if $\p\in \Psupp_R^t(M).$
\end{lemma}

From now on, set $\Lambda_M=\bigcup\limits_{i=0}^d\min \Psupp_R^i(M)$. It is clear that if $\Psupp_R^i(M)$ is closed for all $i$ then $\Lambda_M$ is a finite set.  By Lemma \ref{L:1a}(c), if $R$ is a homomorphic image of a Cohen-Macaulay local ring then $\Lambda_M=\bigcup\limits_{i=0}^{d}\min\Var(\Ann_RH_{\m}^i(M))$ which is a finite set.

\begin{lemma}\label{nxtap} $\Ass_R(M)\subseteq \Lambda_M$. 
\end{lemma}

\begin{proof}	Let $\p\in \Ass_R(M)$. Set $\dim (R/\p)=k.$ Since $\p R_{\p}\in \Ass_{R_{\p}}(M_{\p}),$ we have  $$0\neq H_{\p R_{\p}}^0(M_{\p})=H_{\p R_{\p}}^{k-\dim (R/\p)}(M_{\p}).$$ Hence $\p\in \Psupp_R^k(M)$. Let $\frak q\in\Spec(R)$ such that $\p \supseteq\q$ and $\frak p\neq \frak q$. Then $\dim (R/\q)>k.$  Hence $H_{\q R_{\q}}^{k-\dim (R/\q)}(M_{\q})=0$ and hence $\q \notin \Psupp_R^k(M).$ So, $\p\in \min(\Psupp_R^k(M))\subseteq \Lambda_M.$ \end{proof}

Now we are ready to prove Theorem \ref{T:1}. 

\begin{proof}[Proof of Theorem \ref{T:1}] (a) The method of proving this theorem follows partly the method used in \cite{HT}. Since $\Psupp^i_R(M)$ is closed for all $i$, the set $\Lambda_M=\bigcup_{i=0}^d \min\Psupp^i_R(M)$ is a finite set. 	
	 Let $\p \in \Lambda_M$.  Choose $\q\in\Var(\p+I)$ such that $\displaystyle \docao\left((I+\p)/{\p}\right)=\docao\left(\q/{\p}\right)$. Then we get by \cite[Lemma 1.6]{HT} that
 $$\depth_R(I, M) \leq \depth_R(\q, M) \leq \depth_{R_{\p}}(M_{\p})+\docao\left(\q/{\p}\right) =\depth_{R_{\p}}(M_{\p})+ \docao\left((I+\p)/{\p}\right)$$ for all $\frak p\in \Lambda_M.$ Therefore  $$\depth_R(I,M)\leq \underset{\p\in \Lambda_M}{\min}\left\{ \depth_{R_{\p}}(M_{\p})+ \docao\big((I+\p)/\p\big) \right\}.$$

We prove the converse inequality. Set $\displaystyle n=\underset{\p\in \Lambda_M}{\min}\left\{\depth_{R_{\p}}(M_{\p})+ \docao\big((I+\p)/\p\big) \right\}.$ 
	 We will prove that  $\depth_R(I,M)\geq n$. It is obvious if $n=0$. Assume $n\geq 1.$ By the definition of $n$, we have $$\displaystyle \docao\big((I+\p)/\p\big) \geq n-\depth_{R_{\p}}(M_{\p})$$
	 for all $\p\in \Lambda_M$. 
We claim that there are $n$ elements $y_1, \ldots, y_n\in I$ such that $$\displaystyle\docao\left(((y_1, \ldots, y_n)R+\p)/\p\right) \geq n-\depth_{R_{\p}}(M_{\p})$$ for all $\frak p\in \Lambda_M.$ 
	  Indeed, set $X_1=\{\p\in \Lambda_M\mid  I \nsubseteq \p\}$. Note that if $\frak p\in\Ass_R(M)$ then $\frak p\in \Lambda_M$  by Lemma \ref{nxtap} and $\depth_{R_{\p}}(M_{\p})=0.$ In this case, $\docao\big((I+\p)/\p\big)\geq n\geq 1$ and hence $I\not\subseteq \frak p.$ Therefore 
$$\Ass_R(M)\subseteq \{ \p\in \Lambda_M\mid \docao\big((I+\p)/\p\big)\geq 1 \}\subseteq X_1.$$
Since $\Lambda_M$ is a finite set, there exists by Prime Avoidance an element $y_1 \in I$ such that $y_1\notin \frak p$ for all $\p \in X_1.$ Hence, $\docao\left((y_1)+\p)/\p\right)=1\geq 1-\depth_{R_{\p}}(M_{\p})$ for all $\p \in X_1$. It is clear that $\docao\big((I+\p)/\p\big)=0$ for all $\frak p\in \Lambda_M\setminus X_1.$ Therefore, if $n=1$ then  $$\docao\left((y_1)+\p)/\p\right)=\docao\big((I+\p)/\p\big)\geq 1-\depth_{R_{\p}}(M_{\p})$$ for all $\frak p\in \Lambda_M\setminus X_1.$ So, the claim holds true for $n=1.$ Let $n\geq 2.$ Set $$X_2=\{\frak p\in \Lambda_M\mid \exists \frak q\in\min((y_1)+\frak p), I\not\subseteq \frak q\}.$$ Then we have
	 $$\Ass_R(M) \subseteq \{ \p\in \Lambda_M\mid \docao\big((I+\p)/\p\big)\geq 2 \} \subseteq X_2.$$
	  Since $\Lambda_M$ is a finite set, there exists by Primes Avoidance an element  $y_2\in I$ such that $y_2\notin \p$ for all $\p\in X_2$. Then $ \docao\left(((y_1, y_{2})+\p)/\p\right)=2\geq 2-\depth_{R_{\p}}(M_{\p})$ for all $\p \in X_2.$ It is clear that $\docao\big((I+\p)/\p\big)=1$ for all $\frak p\in \Lambda_M\setminus X_2.$ Therefore if $n=2$ then  $$\docao\left((y_1, y_2)+\p)/\p\right)=\docao\big((I+\p)/\p\big)\geq 2-\depth_{R_{\p}}(M_{\p})$$ for all $\frak p\in \Lambda_M\setminus X_2.$ So, the claim holds true for $n=2.$ Let $n>2$ and assume that there exist $y_1, \ldots , y_{n-1}\in I$ satisfying the requirements. Note that 
	 $$\Ass_R(M) \subseteq \{ \p\in \Lambda_M\mid \docao\big((I+\p)/\p\big)\geq n \} \subseteq X_n,$$
where $X_n=\{\frak p\in \Lambda_M\mid \exists \frak q\in\min ((y_1, \ldots , y_{n-1})+\frak p), I\not\subseteq \frak q\}$.
By the same arguments as in the above, there exists $y_n\in I$ such that $y_1, \ldots , y_n$ satisfy $$\displaystyle\docao\left(((y_1, \ldots, y_n)R+\p)/\p\right) \geq n-\depth_{R_{\p}}(M_{\p})$$ for all $\frak p\in \Lambda_M.$ Thus, the claim is proved.

 Now we prove by induction on $n$ that $y_1, \ldots , y_n\in I$ is an $M$-sequence whenever $$\displaystyle\docao\left(((y_1, \ldots, y_n)R+\p)/\p\right) \geq n-\depth_{R_{\p}}(M_{\p})$$ for all $\frak p\in \Lambda_M.$ Let $n=1$. Suppose in contrary that $y_1\in\frak p$ for some $\p\in\Ass_R(M).$ Then $\depth_{R_{\p}}(M_{\p})=0$ and we have by our hypothesis that $\displaystyle 0=\docao\left(((y_1)+\p)/\p\right)\geq 1.$ This gives a contradiction.    Let $n>1$ and assume that the result holds true for $n-1$. Set $J=(y_1, \ldots , y_n).$ We have by our hypothesis that
$$n-\depth_{R_{\p}}(M_{\p})\leq \docao\big((J+\p)/\p\big)\leq \docao\left(((y_1, \ldots,y_{n-1})+\p)/\p\right)+1$$
for all $\p\in \Lambda_M$. Hence $ \docao\left(((y_1, \ldots,y_{n-1})+\p)/\p\right)\geq (n-1)-\depth_{R_{\p}}(M_{\p})$ for all $\p\in \Lambda_M$. Therefore, we get by induction that $y_1, \ldots, y_{n-1}$ is an $M$-sequence. We will prove that $y_n$ is $M/(y_1, \ldots, y_{n-1})M$-regular. Assume  that there exists $\q \in \Ass_R(M/(y_1, \ldots, y_{n-1})M)$ such that $y_n\in \q$.  We have $\depth_{R_{\q}}(M_{\q})=n-1$. Set $t=\depth_{R_{\q}}(M_{\q})+\dim (R/\q)$. By Lemma \ref{L:1}(a) we have
$\frak q\in \bigcup_{j=0}^t\Psupp_R^j(M)$.
Let $\q'\in \min\bigcup_{j=0}^t\Psupp_R^j(M)$ such that $\q'\subseteq \q$. Then $\q'\in \Lambda_M$. On the other hand, since $\depth_{R_{\q}}(M_{\q})\leq \depth_{R_{\q'}}(M_{\q'})+\docao (\q/\q')$ (see \cite[Lemma 1.6]{HT}), we have
$$t=\depth_{R_{\q}} (M_{\q}) +\dim (R/\q) \leq \depth_{R_{\q'}} (M_{\q'})+\dim (R/\q') \leq t.$$
Hence $$\depth_{R_{\q'}} (M_{\q'})+\dim (R/\q')=t=\depth_{R_{\q}}(M_{\q})+\dim (R/\q).$$ It follows that
$$\depth_{R_{\q'}} (M_{\q'})=\depth_{R_{\q}}(M_{\q})+\dim (R/\q)-\dim (R/\q')=n-1-\docao\left(\q/\q'\right).$$ So, we get that
$$\docao\left(\q/\q'\right) \geq \docao\left((J+\q')/\q'\right) \geq n-\depth_{R_{\q'}}(M_{\q'})=1+\docao\left(\q/\q'\right).$$
This is a contradiction. So, $y_n$ is $M/(y_1, \ldots, y_{n-1})M$-regular and hence $y_1, \ldots, y_n$ is an $M$-sequence. This proves that  $\depth_R(I,M)=\underset{\p\in \Lambda_M}{\min}\left\{ \depth_{R_{\p}}(M_{\p})+ \docao\big((I+\p)/\p\big) \right\}.$

(b) Assume that $R$ is a quotient of a Cohen-Macaulay local ring. By Lemma \ref{L:1a}(b), $\Psupp^i_R(M)=\Var(\Ann_R(H^i_{\m}(M))$ for all $i\leq d$. Therefore, $\Psupp^i_R(M)$ is closed for all $i\leq d$ and 
	$$\Lambda_M= \bigcup_{i=0}^d \min \Var(\Ann_R(H^i_{\m}(M)).$$
	Now  the result follows by the statement (a).
\end{proof}

The following corollary of Theorem \ref{T:1}(a) gives some cases where the formula of $\depth_R(I,M)$ holds true. Following M. Nagata [Na],  $M$ is said to be {\it unmixed} if $\dim (\R/\frak P)=d$ for all prime  ideals $\frak P\in\Ass_{\widehat R}\widehat M.$

\begin{corollary}\label{C:2} $\Psupp^i_R(M)$ is closed for all $i$ in the following cases:

{\rm (a)} $\dim_R (M)\leq 2$;

{\rm (b)} $\dim_R(M)=3$ and $M$ is unmixed.  

In particular, the equality $\depth_R(I,M)=\underset{\p\in \Lambda_M}{\min}\left\{ \depth_{R_{\p}}(M_{\p})+ \docao\big((I+\p)/\p\big) \right\}$ holds true for any proper ideal $I$ of $R$.
\end{corollary}

\begin{proof} By Theorem \ref{T:1}(a), it is enough to prove that $\Psupp^i_R(M)$ is closed for all $i$. It is clear that $\Psupp^0_R(M)\subseteq \{\frak m\}$ so it is closed. If $\Psupp^1_R(M)\subseteq \{\m\}$ then it is closed. Assume that $\Psupp^1_R(M)\nsubseteq \{\m\}$. Let $\frak m\neq \frak p\in \Psupp_R^1(M).$ Then $\dim(R/\frak p)=1$ and $H^0_{\frak pR_{\frak p}}(M_{\frak p})\neq 0.$ Let $\frak P\in\min (\widehat R/\frak p\widehat R).$ Then $\dim (\widehat R/\frak P)=1$ and $$H^0_{\frak P\widehat R_{\frak P}}(\widehat M_{\frak P})\cong H^0_{\frak pR_{\frak p}}(M_{\frak p})\otimes_{R_{\frak p}}\widehat R_{\frak P}\neq 0.$$ Hence $\frak P\in \min \Psupp^1_{\widehat R}(\widehat M).$ Therefore, $\min \Psupp_R^1(M)\subseteq \{\frak P\cap R\mid \frak P\in \min \Psupp^1_{\widehat R}(\widehat M)\}.$ Hence $\Psupp_R^1(M)$ is a finite set. It follows that $\Psupp^1_R(M)$ is closed.
	
 Now we prove (a) and (b). Without loss of generality, we can replace $R$ by $R/\Ann_R(M)$.

(a) Suppose that $\dim_R (M)\leq 2$. Then $\dim (R)\leq 2.$ Hence $R$ is catenary. Therefore, $\Psupp^2_R(M)$ is closed by \cite[Corollary 3.4]{NA}. 

(b)  Suppose that $\dim_R (M)=3$ and $M$ is unmixed. Then $\dim (R)=3.$ Since $M$ is unmixed, $R$ is catenary. Therefore, $\Psupp^3_R(M)$ is closed by \cite[Corollary 3.4]{NA}. Since $M$ is unmixed, $\widehat M$ satisfies the Serre condition $(S_1)$. So, it follows by \cite[2.2.4]{Sch} that $\dim (\widehat R/\Ann_{\widehat R}(H^2_{\frak m\widehat R}(\widehat M))\leq 1.$ So we get by Lemma \ref{L:1a}(b) that $\dim (\widehat R/\frak P)\leq 1$ for all $\frak P\in \Psupp^2_{\widehat R}(\widehat M).$  Since $ M$ is unmixed, $\dim(R/\frak p)=3$ for all $\frak p\in\Ass_R(M)$. Therefore $M$ satisfies the Serre condition $(S_1)$. So we get by \cite[Lemma 4.4]{CNN} that $\dim (R/\frak p)\leq 1$ for all $\frak p\in \Psupp_R^2(M).$ By the same arguments as in the above, we have $$\min \Psupp_R^2(M)\subseteq \{\frak P\cap R\mid \frak P\in \min \Psupp^2_{\widehat R}(\widehat M)\}.$$ Thus, $\Psupp_R^2(M)$ is closed.
\end{proof}

 The following example shows that there exists a Noetherian local ring $(R,\frak m)$ which is not a quotient of Cohen-Macaulay local ring, but $\Psupp^i_R(M)$ is closed for all $i$, and hence the formula of $\depth_R(I,M)$ holds true by Theorem \ref{T:1}(a).

\begin{example} {\rm Let $(R,\frak m)$ be the Noetherian local domain of dimension $2$ constructed by D. Ferrand and M. Raynaud \cite{FR} such that $\widehat R$ has an embedded primes of dimension $1$. Then $R$ is not a quotient of a Cohen-Macaulay local ring. For any finitely generated $R$-module $M$, all the pseudo supports of $M$ are closed and the formula of $\depth_R(I,M)$ holds true by Corollary \ref{C:2}(a)}.
\end{example} 

In the next example, we will show that there exists a Noetherian local domain $R$ which has non-closed pseudo supports, but the formula of $\depth_R(I,R)$ in Theorem \ref{T:1}(a) still holds true.
	 
\begin{example} \label{E:1} {\rm Let $(R,\m)$ be the $3$-dimensional Noetherian local domain such that
		
 (a) $R$ is not canetary, but all formal fibers of $R$ are Cohen-Macaulay;
		
 (b) $\Psupp_R^0(R)=\emptyset$, $\Psupp_R^1(R)=\{\m\}$;
		
 (c) $\Psupp_R^2(R)= \{\m\}\cup \{\frak p\in \Spec(R)\mid \docao(\p)+\dim(R/\p)=2\}$ which is not closed;

 (d) $\Psupp_R^3(R)=\{\frak p\in \Spec(R)\mid \docao(\p)+\dim(R/\p)=3\}$ which is not closed.
 
Such a domain exists by M. Brodmann and R. Y. Sharp \cite[Example 3.2]{BS1}.
Let $I$ be a proper ideal of $R$. Set $\Lambda_R=\bigcup_{i=0}^3 \min\Psupp^i_R(R)$ and set $\Lambda_R':=\{\m, 0\}$. Then $\Lambda_R' \subseteq \Lambda_R$ and
$$\depth_R(I,R)=\underset{\p\in \Lambda_R}{\min}\left\{\depth(R_{\p})+ \docao\big((I+\p)/\p\big) \right\}=\underset{\Lambda_R'}{\min}\left\{\depth(R_{\p})+ \docao\big((I+\p)/\p\big) \right\}.$$
}
\end{example}
\begin{proof} Set $U=\left\{\p\in \Spec(R)\mid \docao(\p)+\dim(R/\p)=2\right\}.$ It is clear that $0\in \Psupp_R^3(R).$ Therefore 
$\Lambda_R =\{\m, 0\} \cup \min U.$ 
 Note that
$$\depth_R(I,R)\leq \underset{\p\in \Lambda_R}{\min}\left\{ \depth(R_{\p})+ \docao\big((I+\p)/\p\big) \right\}.$$
We have $\depth_R(I, R)\leq \depth(R)=1.$ If $\depth_R(I, R)=0$ then $I=0$ since $R$ is a domain. In this case, we choose $\p=0$, then $\frak p\in\Lambda_R'$ and we have $\depth(R_{\p})+ \docao\big((I+\p)/\p\big)=0$. So, we assume that $\depth_R(I, R)=1$. In this case we choose  $\p=\m$, then $\frak p\in\Lambda_R'$ and we have $$\depth(R_{\p})+ \docao\big((I+\p)/\p\big)=1.$$ 
Therefore the result follows. In this example, $\Lambda_R$ is an infinite set and $\Lambda_R'$ is a finite set.
\end{proof}

Note that  $R$ is a quotient of a Cohen-Macaulay local ring if and only if $R$ is universally catenary and all its formal fibers are Cohen-Macaulay, cf. \cite[Corollary 1.2]{Kawasaki}.  The next example shows that there exists a universally catenary local domain $(R,\frak m)$ with a non-Cohen-Macaulay fiber and the formula of $\depth_R(I,M)$ in Theorem \ref{T:1}(b) does not hold true. Before doing that, we recall the secondary representation of Artinian modules introduced by I. G. Macdonald \cite{Mac}. Let $A$ be an Artinian $R$-module. Then $A$ has a minimal secondary representation $A=A_1+\ldots +A_n$, where each $A_i$ is $\frak p_i$-secondary, $A_i$ is not redundant and $\frak p_i\neq \frak p_j$ for all $i\neq j.$ The set $\{\frak p_1, \ldots , \frak p_n\}$ is independent of the choice of minimal secondary representation of $A$. This set is called the set of {\it attached primes} of $A$ and denoted by $\Att_R(A).$ Note that $\min\Att_R(A)=\min\Var(\Ann_R(A)),$ in particular, $A\neq 0$ if and only if $\Att_R(A)\neq \emptyset$. Moreover, by \cite[8.2.8 and 8.2.5]{BS}, $A$ has a natural structure of an Artinian $\widehat R$-module and with this structure we have
$$\Att_R(A)=\{\frak P\cap R\mid \frak P\in\Att_{\widehat R}(A)\}.$$  
 \begin{example}\label{E:2}{\rm Let $(R, \m)$ be the $3$-dimensional Noetherian local domain such that 
 		
 		(a) $\widehat R \cong \mathbb Q[[V_1, V_2, X, Y]]/(V_1 V_2)\cap (V_1^2, V_2^2)$, where $V_1, V_2, X, Y$ are independent indeterminates over $\Bbb Q$;
 		
 		(b) $\Psupp^0_R(R)=\Psupp^1_R(R)=\emptyset$,  $\Psupp^3_R(R)=\Spec(R)$;
 		
 		(c) $\Psupp^2_R(R)=\{\p\in \Spec(R)\mid \depth (R_{\p})=\dim(R_{\p})-1\}=\nCM(R)$ which is not closed,  where $\nCM(R)$ is the non-Cohen-Macaulay locus of $R$.
 		
 		Such a domain exists,  see \cite[Example 3.1]{BS}. Then $R$ is universally catenary, $R$ has a non-Cohen-Macaulay formal fiber, $\nCM(R)$ is a non-closed infinite set, and the formula of $\depth_R(I,R)$ in Theorem \ref{T:1}(b) does not hold true for all ideals $I$ with $\Rad(I)\in\nCM(R).$}
 \end{example}
 
 \begin{proof} $R$ is universally catenary as $\widehat R$ is equidimensional. Set  $\frak P=(V_1, V_2)\widehat R$. Since $\frak P$ is an embedded prime of $\widehat R$, it follows that $R$ has a non-Cohen-Macaulay formal fiber. Let $\Lambda_R=\bigcup_{i=0}^3 \min \Var(\Ann_R(H^i_{\m}(R)))$ be defined as in Theorem \ref{T:1}(b). 
 	
 	Since $R$ is catenary, we have  $$\Var(\Ann_R(H^3_{\m}(R))=\Psupp^3_R(R)=\Spec(R)$$ by \cite[Corollary 3.4]{NA}. Hence $\min \Var(\Ann_R(H^3_{\m}(R))=\{0\}$. Since $\mathfrak{P}\in \Ass(\widehat R)$ and $\dim(\widehat R/\mathfrak{P})=2,$ it follows by \cite[11.3.9]{BS} that $\mathfrak{P}\in \Att_{\widehat R} (H^2_{\m}(R))$. Since $R$ is a domain and $\mathfrak{P}\cap R \in \Ass(R)$ by \cite[Lemma 3.4]{Nh}, we have $0=\mathfrak{P}\cap R$. Hence $0\in \Att_{ R} (H^2_{\m}(R))$. So, $\min \Var(\Ann_R(H^2_{\m}(R))=\{0\}$. Therefore, $\Lambda_R=\{0\}$. Now we show that the formula of $\depth (I, R)$ does not hold true for all ideal $I$ such that $\Rad(I)\in\nCM(R).$  We divide into two cases. 
 	
 	$\bullet$ Suppose that $\Rad(I)=\m$. Then $$\depth_R(I,R)=2<\underset{\p\in \Lambda_R}{\min}\left\{ \depth(R_{\p})+ \docao\big((I+\p)/\p\big) \right\} =3.$$
 	
 	$\bullet$ Suppose that $\Rad(I)=\frak q\in\nCM(R)\setminus\{\frak m\}.$ Then $\dim(R_{\q})=2$ and $\depth (R_{\q})=1.$ Hence $\docao (\frak q)=2$ and $\depth_R(\frak q, R)=1.$ Therefore $$\depth_R(I,R)=1<\underset{\p\in \Lambda_R}{\min}\left\{ \depth(R_{\p})+ \docao\big((I+\p)/\p\big) \right\} =2.$$
 	This means the formula of $\depth_R(I,R)$ in Theorem \ref{T:1}(b) does not hold true for all ideals $I$ with $\Rad(I)\in\nCM(R).$ 	
 \end{proof}

 Finally, we give an example to show that there exists a Noetherian  local domain $(R,\frak m)$ such that all formal fibers of $R$ are Cohen-Macaulay,  $R$ is not universally catenary and the formula of $\depth_R(I,M)$ in Theorem \ref{T:1}(b) does not hold true.
 
 \begin{example}\label{E:3}{\rm
 		Let $\mathbb Q[[x,y,z,w,t]]$ be the ring of formal power series in $4$ variables over $\mathbb Q$. Then there exists by \cite{CL} a $3$-dimensional Noetherian local domain $(R,\m)$ such that
 		
 		(a) $\widehat R\cong \mathbb Q[[x,y,z,t]]/(x)\cap (y,z)$;
 		
 		(b) $\tau^{-1}(0)=\{x\widehat R, (y,z)\widehat R\}$; 
 		
 		(c) $\tau^{-1}(\frak P\cap R)=\{ \frak P\}$ for all $\frak P \in \Spec(\widehat R)\setminus \tau^{-1}(0)$,
 		where $\tau: \Spec(\widehat R) \rightarrow \Spec(R)$ is the map induced by the natural homomorphism $R\rightarrow \widehat R$.
 		
 		Then all formal fibers of $R$ are Cohen-Macaulay, $R$ is not universally catenary, $\nCM(R)$ is a closed subset of $\Spec(R)$ with $\dim \nCM(R)=1$ and the formula of $\depth_R(I,R)$ in Theorem \ref{T:1}(b) does not hold true for all ideals $I$ of $R$ such that $\Rad(I)\in\nCM(R).$}
 \end{example}
 
 \begin{proof} Since $\widehat R$ is not equidimensional, $R$ is not universally catenary. It follows by (b) and (c) that all formal fibers of $R$ are Cohen-Macaulay. We have $H_{\m}^0(R)=0$. Since $R$ is a domain, we get by \cite[7.3.2]{BS} that $0\in \Att_R(H_{\m}^3(R)).$ Hence $\min \Var(\Ann_R(H_{\m}^3(R)))=\{0\}.$ Set $\frak P=(y,z)\widehat R$. Then $\frak P\in \Ass(\widehat R)$ and $\dim (\widehat R/\frak P)=2$. So we get by \cite[11.3.9]{BS} that $\frak P\in \Att_{\widehat R}(H_{\m}^2(R)).$ Since $R$ is a domain and $\frak P\cap R\in \Ass(R)$ by \cite[Lemma 3.4]{Nh}, we have $0\in \Att_{ R}(H_{\m}^2(R)).$ Therefore, $\min \Var(\Ann_R(H_{\m}^2(R)))=\{0\}.$ From the exact sequence
 	$$0\rightarrow \widehat R\rightarrow \widehat R/x\widehat R \oplus \widehat R/(y,z)\widehat R\rightarrow  \widehat R/(x,y,z)\widehat R\rightarrow 0$$
 	we have
 	$H_{\m\widehat R}^1(\widehat R) =0$. It follows that $H_{\m}^1(R)=0.$ Therefore, 
 	$$\Lambda_R=\bigcup_{i=0}^3 \min \Var(\Ann_R(H^i_{\m}(R)))=\{0\}.$$
 	
 	Since all formal fibers of $R$ are Cohen-Macaulay, $\nCM(R)$ is closed. It is clear that $\frak m\in\nCM(R).$ Set $\frak Q=(x,y,z)\widehat R$. Then $\widehat R_{\frak Q}$ is not unmixed. Therefore $\widehat R_{\frak Q}$ is not Cohen-Macaulay. We have $\dim (\widehat R/\frak Q)=1.$ Set $\frak q=\frak Q\cap R.$ By (c), $\frak Q$ is the unique prime ideal of $\widehat R$ such that $\frak Q\cap R=\frak q.$ Hence $\dim(R/\frak q)=1$ and $R_{\frak q}$ is not Cohen-Macaulay.  So, $\dim\nCM(R)\geq 1.$ Let  $\frak q\in\nCM(R)\setminus\{\frak m\}.$ As $R$ is a domain, $\dim(R_{\frak q})=2$ and $\depth (R_{\frak q})=1.$ Hence $\dim(R/\frak p)=1.$ Therefore $\dim\nCM(R)\leq 1.$  
 	
 	Now, by the same arguments as in Example \ref{E:2}, we can show that $$\depth_R(I,R)+1=\underset{\p\in \Lambda_R}{\min}\left\{ \depth(R_{\p})+ \docao\big((I+\p)/\p\big) \right\}$$ for all ideals $I$ of $R$ such that $\Rad(I)\in\nCM(R).$  
 \end{proof}


\begin{thebibliography}{99}
\bibitem [An]{An}	T. N. An, \textit{On the  attached primes and Shifted Localization Principle for  local cohomology modules}, Algebra Colloq.,  \textbf{20} (2013),  671-680.


\bibitem [BS]{BS}M.  Brodmann   and R. Y. Sharp, {\it  Local cohomology: an algebraic introduction with geometric applications}, Cambridge University Press, 1998.

\bibitem [BS1]{BS1} M. Brodmann and R. Y. Sharp, {\it On the dimension and multiplicity of local cohomology modules}, Nagoya Math. J., {\bf 167} (2002), 217-233.

\bibitem [CL]{CL} P. Charters and S. Loepp, {\it Semilocal generic formal fibers}, J. Algebra, {\bf 278} (2004), 370-382.

\bibitem [CNN]{CNN} N. T. Cuong, L. T. Nhan, N. T. K. Nga and {\it On pseudo supports and non Cohen-Macaulay locus of a finitely generated module}, J. Algebra, \textbf{323} (2010), 3029-3038.


\bibitem [FR]{FR} D. Ferrand  and M. Raynaud,  {\it Fibres formelles d'un anneau local Noetherian,} Ann. Sci. E'cole Norm. Sup.,  {\bf (4)} (1970), 295-311.

\bibitem [HT]{HT} C. Huneke and V. Trivedi, \textit{The height of ideals and regular sequences}, Manuscripta Mathematica, \textbf{93} (1997), 137-142.


\bibitem [Kaw]{Kawasaki} T. Kawasaki, \textit{On arithmetic Macaulayfication of local rings},  Trans. AmSseter. Math. Soc. \textbf{354} (2002), 123–149.

\bibitem [Mac]{Mac} I. G. Macdonald, {\it  Secondary representation of modules    over a commutative ring,}  Symposia Mathematica, {\bf 11} (1973), 23-43.

\bibitem [Mat]{Mat} H. Matsumura, {\it Commutative ring theory}, Cambridge University Press, 1986.
\bibitem [Na]{Na} M.  Nagata,  ``Local rings", Interscience, New York, 1962.

 \bibitem [Nh]{Nh} L. T. Nhan, {\it On generalized regular  sequences and the  finiteness for associated primes of local cohomology modules,}  Comm. Algebra,  {\bf 33} (2005),  793-806.

\bibitem [NA]{NA} L. T. Nhan and T. N. An,  {\it On the unmixedness and the universal catenaricity of local rings and local cohomology modules,}  J. Algebra, {\bf 321} (2009), 303-311.

\bibitem [NA1]{NA1} L. T. Nhan and T. N. An,  {\it On the catenaricity of Noetherian local rings and quasi unmixed Artinian modules,} Comm. Algebra, \textbf{38} (2010), 3728-3736. 

\bibitem [Sch]{Sch} P. Schenzel, ``Dualisierende Komplexe in der lokalen Algebra und Buchsbaum Ringe", Lecture Notes in Mathematics \textbf{907}, Springer-Verlag, 1982.
 \end{thebibliography}
\end{document}